\documentclass[11pt]{article}%{amsart}
\usepackage[all]{xypic}

\newtheorem{theorem}{Theorem}[section]
\newtheorem{corollary}[theorem]{Corollary}
\newtheorem{lemma}[theorem]{Lemma}
\newtheorem{proposition}[theorem]{Proposition}
\newtheorem{definition}[theorem]{Definition}
\newtheorem{example}[theorem]{Example}
\newtheorem{remark}[theorem]{Remark}

\newenvironment{proof}{{\bf Proof}}{{\hfill $ \Box $}\vskip 4mm}

\def\Box{\mbox{$\sqcap\!\!\!\!\sqcup$}}

\def\bea{\begin{eqnarray*}}
\def\eea{\end{eqnarray*}}

\begin{document}
\title{Frobenius algebras of corepresentations: gradings.}
\author{ S.
D\u{a}sc\u{a}lescu, C. N\u{a}st\u{a}sescu and L. N\u{a}st\u{a}sescu \\[2mm]
University of Bucharest, Facultatea de Matematica,\\ Str.
Academiei 14, Bucharest 1, RO-010014, Romania,\\
 e-mail: sdascal@fmi.unibuc.ro, Constantin\_nastasescu@yahoo.com,
 lauranastasescu@gmail.com}

\date{}
\maketitle

\begin{abstract}
We consider Frobenius algebras in the monoidal category of right
comodules over a Hopf algebra $H$. If $H$ is a group Hopf algebra,
we study a more general Frobenius type property and uncover the
structure of graded Frobenius algebras. Graded
symmetric algebras are also investigated.
\end{abstract}

\section{Introduction and preliminaries}

Frobenius algebras originate in the work of F. G. Frobenius on
representation theory of finite groups. Their study was initiated
by Brauer, Nesbitt and Nakayama, and then continued by
Dieudonn$\rm \acute{e}$, Eilenberg, Azumaya, Thrall, Kasch, etc.,
uncovering a rich representation theory. Also, Frobenius algebras
have played a role in Hopf algebra theory (any finite dimensional
Hopf algebra is Frobenius), cohomology rings of compact oriented
manifolds, orbifold theories, solutions of the quantum Yang-Baxter
equation, Jones polynomials, topological quantum field theory; see
for instance \cite{lam}, \cite{kadison}. A finite dimensional
algebra $A$ over a field $k$ is Frobenius if $A\simeq A^*$ as left
$A$-modules. Abrams proved in \cite{abrams1}, \cite{abrams2} that
$A$ is Frobenius if and only if there is a coalgebra structure
$(A,\delta, \varepsilon_A)$ on the space $A$ such that $\delta$ is
a morphism of $A,A$-bimodules. Since Abrams' characterization
makes sense in any monoidal category, this provided the right way
to define the concept of Frobenius algebra in  such a framework.
The study of Frobenius algebras in monoidal categories was
initiated and developed by M$\rm \ddot{u}$ger \cite{muger}, Street
\cite{street}, Fuchs and Stigner \cite{fuchs}, Yamagami \cite{y}.
It is a challenging problem to understand the significance of
Frobenius algebras in monoidal categories other than categories of
vector spaces. M$\rm \ddot{u}$ger investigated in \cite{muger}
Frobenius algebras arising in category theory and subfactor
theory. Fuchs, Schweigert and Stigner \cite{fss} considered
certain Frobenius algebras in the monoidal category of
$H$-bimodules, where $H$ is a finite dimensional factorizable
ribbon Hopf algebra. In Section \ref{scorepresentations} we
consider the monoidal category ${\cal M}^H$ of right comodules
(i.e. corepresentations) over a Hopf algebra $H$. We do not
require that the antipode $S$ is bijective, thus ${\cal M}^H$ has
left duals, but not necessarily right duals. If $A$ is a finite
dimensional algebra in this category, i.e. a right $H$-comodule
algebra, then both $A$ and its dual space $A^*$ have natural
structures of objects in the category ${\cal M}^H_A$ of right
Doi-Hopf modules. We say that $A$ is right $H$-Frobenius if these
two objects are isomorphic. On the other hand, $A^*$ has a natural
structure of a left Doi-Hopf module over the right $H$-comodule
algebra $A^{(S^2)}$, which is just $A$ as an algebra, and has the
$H$-coaction shifted by $S^2$. We say that $A$ is left
$H$-Frobenius if $A^*$ and $A^{(S^2)}$ are isomorphic as such left
Doi-Hopf modules. We give equivalent characterizations of these
two Frobenius properties. In fact, the right $H$-Frobenius
property is not new, it is equivalent to $A$ being a Frobenius
algebra in the category ${\cal M}^H$. The left $H$-Frobenius
property seems to be specific to the category of $H$-comodules,
there is no obvious such concept in an arbitrary monoidal
category. We prove that $A$ is left $H$-Frobenius if and only if
$A^{(S^2)}$ is a Frobenius algebra in the monoidal category ${\cal
M}^H$. Also, we show that if $A$ is right $H$-Frobenius, then $A$
is left $H$-Frobenius; if $S$ is injective,  the converse also
holds.

 In the rest of the paper we specialize the study of Frobenius algebras in ${\cal M}^H$ to the
case where $H$ is a group Hopf algebra $kG$ of an arbitrary group
$G$. In this case the corepresentations of $H$ are the $G$-graded
vector spaces, and an algebra in this category is just a
$G$-graded algebra. A $G$-graded algebra $A$ is called graded
Frobenius if it is right (or equivalently left) $kG$-Frobenius. In
fact we define a more general concept in Section
\ref{ssigmaFrobenius}. For $\sigma \in G$, we say that $A$ is
$\sigma$-graded Frobenius if $A^*$ is isomorphic to the suspension
$A(\sigma)$ of $A$ in the category $A-gr$ of graded left
$A$-modules. We study basic properties of $\sigma$-graded
Frobenius algebras and give several characterizations of them. One
of our main results says that a finite dimensional $G$-graded
algebra $A$ is $\sigma$-graded Frobenius if and only if
$A_{\sigma}\simeq A_e^*$ as left $A_e$-modules, and $A$ is left
$\sigma$-faithful. Here $A_\sigma$ denotes the homogeneous
component of degree $\sigma$ of $A$, and $e$ is the neutral
element of $G$. In particular, this
 gives the structure of Frobenius algebras in the
category of graded vector spaces. These are the finite dimensional
$G$-graded algebras $A$ which are left $e$-faithful and such that
$A_e$ is a Frobenius algebra. In particular, graded semisimple
algebras are graded Frobenius.
 We
give examples (in Section \ref{sectionexamples}) of Frobenius
algebras which are not $\sigma$-graded Frobenius for any $\sigma$.
However, if $A_e$ is a local ring, and $A$ is Frobenius, we show
that $A$ is $\sigma$-graded Frobenius for some $\sigma$.

Among graded Frobenius algebras there are some objects with more
symmetry: the
 graded symmetric algebras, which are just the symmetric algebras
 in the sovereign category of graded vector spaces, see
 \cite{fuchs}. In Section \ref{ssimetric} we show that a graded division algebra $\Delta$ is
 graded symmetric, provided that the center of $\Delta_e$ consists
 of the central elements of $\Delta$ of degree $e$. In Section \ref{sectionexamples} we give
 several examples to illustrate the concepts of $\sigma$-graded
 Frobenius, graded symmetric, Frobenius and symmetric, and
 connections between them. In particular we show that any matrix
 algebra (over a field), endowed with a grading such that the matrix units are homogeneous elements, is graded
 symmetric. Using the description of abelian gradings on matrix algebras over an algebraically closed field, given by Bahturin, Sehgal and Zaicev
 \cite{bsz}, we show that any such grading is graded symmetric.

 Finally, in Section \ref{sfunctors} we discuss the concept of graded Frobenius in relation
 to Frobenius functors. The property of being a Frobenius algebra
 is not Morita invariant. However a matrix algebra
 $M_n(A)$ is Frobenius whenever $A$ is so. We show that the
 converse also holds: if $M_n(A)$ is Frobenius, then so is $A$. We
 characterize graded Frobenius algebras as those finite
 dimensional graded algebras $A$ for which the functor
 $U:A-gr\rightarrow k-gr$, forgetting the $A$-action, is a
 Frobenius functor. We connect the graded Frobenius property for
 $A$ and the Frobenius property for the functor
 $(-)_e:A-gr\rightarrow A_e-mod$. Finally, we give a new proof for
 a result of Bergen \cite{bergen} stating that if $H$ is a finite
 dimensional Hopf algebra acting on the finite dimensional algebra
 $A$, then the smash product $A\# H$ is Frobenius if and only if
 so is $A$.

 For definitions and notation we refer to \cite{DNR}, \cite{mo}
 for coalgebras and Hopf algebras, and to \cite{nvo} for graded
 algebras. All algebras, coalgebras and graded algebras are
 considered over a field $k$.

\section{Frobenius algebras in categories of comodules}
\label{scorepresentations}

Let $C$ be a coalgebra, and let $M\in {\cal M}^C$ be a finite
dimensional comodule with comodule structure map
$\rho:M\rightarrow M\otimes C$. Then $M$ is a left $C^*$-module in
the standard way, and then $M^*$ is a right $C^*$-module. By
\cite[Lemma 2.2.12]{DNR} this is a rational right $C^*$-module,
i.e. its module structure comes from a left $C$-comodule
structure. Since we need the details of the proof, we indicate how
this comes out. Consider the natural isomorphism $\gamma: M\otimes
C\otimes M^*\rightarrow Hom(M,M\otimes C)$, $\gamma (m\otimes
c\otimes m^*)(x)=m^*(x)m\otimes c$ for any $m,x\in M,m^*\in
M^*,c\in C$, and look at the inverse image of $\rho$, i.e.
$\rho=\gamma (\sum_im_i\otimes c_i\otimes m_i^*)$. Then
$\rho(m)=\sum _i m_i^*(m)m_i\otimes c_i$ for any $m\in M$, and a
direct computation shows that $m^*c^*=\sum_ic^*(c_i)m^*(m_i)m_i^*$
for any $m^*\in M^*$ and $c^*\in C^*$, i.e. the right $C^*$-module
structure of $M^*$ comes from a left $C$-comodule structure given
by $\mu:M^*\rightarrow C\otimes M^*, \mu (m^*)=\sum
_im^*(m_i)c_i\otimes m_i^*$.

Now let $H$ be a Hopf algebra and let $A$ be a right $H$-comodule
algebra with $H$-comodule structure denoted by $a\mapsto \sum
a_0\otimes a_1$. Let $M\in \, _A{\cal M}^H$ be a finite
dimensional left $(A,H)$-Doi-Hopf module, i.e. it is a left
$A$-module and a right $H$-comodule with $\rho:M\rightarrow
M\otimes H, \rho(m)=\sum m_0\otimes m_1$, such that $\rho
(am)=\sum a_0m_0\otimes a_1m_1$ for any $a\in A, m\in M$. As
above, there exists $\sum_i m_i\otimes h_i\otimes m_i^*\in
M\otimes H\otimes M^*$ such that $\rho(m)=\sum _i
m_i^*(m)m_i\otimes h_i$ for any $m\in M$, and $M^*$ is a left
$H$-comodule by $m^*\mapsto \sum_i m^*(m_i)h_i\otimes m_i^*$.
Using the antipode $S$ of $H$, $M^*$ becomes a right $H$-comodule
via $\nu:M^*\rightarrow M^*\otimes H$, $\nu (m^*)=\sum
m^*_{(0)}\otimes m^*_{(1)}=\sum_im^*(m_i)m_i^*\otimes S(h_i)$.
Note that since $M$ is a left $(A,H)$-Doi-Hopf module, we have
that
\begin{equation}\label{propform}
\sum_im_i^*(am)m_i\otimes h_i=\sum_im_i^*(m)a_0m_i\otimes a_1h_i
\end{equation}
for any $a\in A$ and $m\in M$.

\begin{proposition} \label{dualDoiHopf}
If $M\in \, _A{\cal M}^H$ is  finite dimensional, then $M^*\in \,
{\cal M}_A^H$, i.e. $M^*$ is a right $(A,H)$-Doi-Hopf module with
the right $H$-comodule structure $\nu$.
\end{proposition}
\begin{proof}
We have to show that $\nu (m^*a)=\sum m^*_{(0)}a_0\otimes
m^*_{(1)}a_1$. Taking into account the comodule structure formula,
this is equivalent to
\begin{equation} \label{propeq1}
\sum_im^*(am_i)m_i^*\otimes S(h_i)=\sum_im^*(m_i)(m_i^*a_0)\otimes
S(h_i)a_1
\end{equation}
Since $M^*\otimes H\simeq Hom(M,H)$, equation (\ref{propeq1}) is
equivalent to the fact that for any $m\in M$ we have that
\begin{equation} \label{propeq2}
\sum_im^*(am_i)m_i^*(m)S(h_i)=\sum_im^*(m_i)(m_i^*a_0)(m)S(h_i)a_1
\end{equation}
But we have that \bea
\sum_im^*(m_i)(m_i^*a_0)(m)S(h_i)a_1&=&\sum_im^*(m_i)m_i^*(a_0m)S(h_i)a_1\\
&=&\sum_im^*(a_0m_i)m_i^*(m)S(a_1h_i)a_2\;\;\;\; (\mbox{by
}(\ref{propform})\mbox{ for }a_0\mbox{ and }m)\\
&=&\sum_im^*(am_i)m_i^*(m)S(h_i)\eea Thus (\ref{propeq2}) is
satisfied, and this ends the proof.
\end{proof}

If $A$ is a right $H$-comodule algebra, we denote by $\cdot$ the
left $H^*$-action on $A$ associated to the right $H$-coaction. We
define a new object $A^{(S^2)}$ as being just $A$ as an algebra,
and with a right $H$-coaction $a\mapsto \sum a_0\otimes S^2(a_1)$.
Then $A^{(S^2)}$ is a right $H$-comodule algebra. We denote by $*$
the associated left $H^*$-action on $A$. In a way similar to
Proposition \ref{dualDoiHopf}, one can see that if $M\in {\cal
M}_A^H$, a right Doi-Hopf module, is finite dimensional, then
$M^*\in \, _{A^{(S^2)}}{\cal M}^H$, where the $H$-coaction on
$M^*$ is given by the same formula as above for left Doi-Hopf
modules.

Now assume that $A$ is a finite dimensional right $H$-comodule
algebra. As a particular case of the discussion above, we see that
there exists $\sum_i a_i\otimes h_i\otimes a_i^*\in A\otimes
H\otimes A^*$ such that $\sum a_0\otimes
a_1=\sum_ia_i^*(a)a_i\otimes h_i$ for any $a\in A$. Then $A^*\in
{\cal M}^H$ with coaction $a^*\mapsto \sum a^*_0\otimes
a^*_1=\sum_ia^*(a_i)a_i^*\otimes S(h_i)$. We keep this notation
for the rest of the section.

We can regard $A$ as a left $(A,H)$-Doi-Hopf module, and also as a
right $(A,H)$-Doi-Hopf module. The first structure induces a
structure of a right $(A,H)$-Doi-Hopf module on $A^*$, while the
second one induces a structure of a left $(A^{(S^2)},H)$-Doi-Hopf
module on $A^*$.

\begin{definition}
The finite dimensional right $H$-comodule algebra $A$ is called\\
$\bullet$ left $H$-Frobenius if $A^{(S^2)}\simeq A^*$ in
$_{A^{(S^2)}}{\cal M}^H$.\\
$\bullet$ right $H$-Frobenius if $A\simeq A^*$ in ${\cal M}_A^H$.
\end{definition}

 The following characterizes the left $H$-Frobenius property. The
 first four equivalent conditions are
in the spirit of the classical ones for Frobenius algebras, taking
also care of the $H$-coaction. The last condition shows the
connection to the concept of Frobenius algebra in the category of
corepresentations of $H$.

\begin{theorem} \label{charFrobcomod}
Let $A$ be a finite dimensional right $H$-comodule algebra.
The following assertions are equivalent.\\
{\rm (1)} $A$ is left $H$-Frobenius.\\
{\rm (2)} There exists a non-degenerate associative bilinear form
$B:A\times A\rightarrow k$ with the property that
$B(b,(h^*S^2)\cdot
a)=B((h^*S)\cdot b,a)$ for any $a,b\in A$ and any $h^*\in H^*$.\\
{\rm (3)} $A$ has a hyperplane $\cal H$ which does not contain any
non-zero left ideal of $A$, and $(h^*S^2)\cdot
A\subseteq {\cal H}$ for any $h^*\in H^*$ with $h^*(1)=0$.\\
{\rm (4)} $A$ has a hyperplane $\cal H$ which does not contain any
non-zero subobject of $A^{(S^2)}$ in $_{A^{(S^2)}}{\cal M}^H$, and
$(h^*S^2)\cdot
A\subseteq {\cal H}$ for any $h^*\in H^*$ with $h^*(1)=0$.\\
{\rm (5)} $A^{(S^2)}$ is a Frobenius algebra in the monoidal
category ${\cal M}^H$.
\end{theorem}
\begin{proof}
If $\theta:A\rightarrow A^*$ is a linear map, let $B:A\times
A\rightarrow k$ be the bilinear map defined by $B(a,b)=\theta
(b)(a)$ for any $a,b\in A$.  We have that $\theta (h^**
a)(b)=B(b,h^** a)=B(b,(h^*S^2)\cdot a)$, and \bea (h^*\theta(a))(b)&=&\sum_i h^*(S(h_i))\theta (a)(a_i)a_i^*(b)\\
&=& \sum_i B(a_i,a)a_i^*(b)h^*(S(h_i))\\
&=&B(\sum_i a_i^*(b)h^*(S(h_i))a_i,a)\\
&=&B((h^*S)\cdot b,a)\eea

This shows that $\theta$ is a morphism of left $H^*$-modules, or
equivalently, of right $H$-comodules, if and only if
$B(b,(h^*S^2)\cdot a)=B((h^*S)\cdot b,a)$ for any $a,b\in A$ and
any $h^*\in H^*$. By the proof of \cite[Theorem 3.15]{lam},
$\theta$ is an isomorphism of left $A$-modules if and only if $B$
is associative and non-degenerate. Now it is clear that (1) and
(2) are equivalent.

Now we show that $(2)\Rightarrow (3)$. Let ${\cal H}=\{ a\in A|\;
B(1,a)=0\}$. If $I$ is a left ideal of $A$ such that $I\subseteq
{\cal H}$, then for any $a\in I$ and $x\in A$ we have that
$B(x,a)=B(1,xa)=0$, showing that $a=0$. Thus $I$ must be 0. On the
other hand, if $a\in A$ and $h^*\in H^*$, one has \bea
B(1,(h^*S^2)\cdot a-h^*(1)a)&=&B(1,(h^*S^2)\cdot a)-h^*(1)B(1,a)\\
&=&B((h^*S)\cdot 1,a)-h^*(1)B(1,a)\\
&=&B((h^*S)(1)1,a)-h^*(1)B(1,a)\\
&=&0\eea so (3) holds.

$(3)\Rightarrow (4)$ is obvious.

 $(4)\Rightarrow (1)$  Since for an arbitrary $h^*\in H^*$ we have that
$(h^*S^2-h^*(1)\varepsilon)(1)=0$, the last condition in (3) is
equivalent to $(h^*S^2)\cdot a-h^*(1)a\in {\cal H}$ for any $a\in
A$ and $h^*\in H^*$. Let $\pi:A\rightarrow k$ be a linear map
whose kernel is $\cal H$. Define $\theta:A^{(S^2)}\rightarrow A^*$
by $\theta (b)(a)=\pi (ab)$ for any $a,b\in A$, and let
$B(a,b)=\theta (b)(a)$ be the associated bilinear form. Then
$\theta (ay)(x)=\pi (xay)=\theta(y)(xa)=(a\theta(y))(x)$, so
$\theta$ is a morphism of left $A$-modules.

On the other hand, if $h^*\in H^*$ and $a,b\in A$, we have

\bea ((h^*S)\cdot b)a-b((h^*S^2)\cdot a)&=&\sum
(h^*S)(b_1)b_0a-b((h^*S^2)\cdot a) \\
&=&\sum (h^*S)(b_1a_1S(a_2))b_0a_0-b((h^*S^2)\cdot a) \\
&=&\sum (S(a_1)\rightharpoonup (h^*S))\cdot (ba_0)-b((h^*S^2)\cdot a)\\
&=&\sum ((S(a_1)\rightharpoonup (h^*S))\cdot
(ba_0)-(S(a_1)\rightharpoonup (h^*S))(1)ba_0) \\
&&+\sum (S(a_1)\rightharpoonup (h^*S))(1)ba_0-b((h^*S^2)\cdot a)\\
&=&\sum ((S(a_1)\rightharpoonup (h^*S))\cdot
(ba_0)-(S(a_1)\rightharpoonup (h^*S))(1)ba_0) \\
&&+\sum (h^*S^2)(a_1)ba_0-\sum (h^*S^2)(a_1)ba_0\\
&=&\sum ((S(a_1)\rightharpoonup (h^*S))\cdot
(ba_0)-(S(a_1)\rightharpoonup (h^*S))(1)ba_0)\in {\cal H} \eea

This shows that  $\pi (((h^*S)\cdot b)a)=\pi (b((h^*S^2)\cdot
a))$, which means that $\theta$ is $H^*$-linear, or equivalently,
right $H$-colinear. On the other hand $I=Ker(\theta)$, which is  a
subobject of $A^{(S^2)}$ in $_{A^{(S^2)}}{\cal M}^H$, so it must
be zero, showing that $\theta$ is injective, hence it is an
isomorphism.

$(1)\Rightarrow (5)$ Let $\theta:A^{(S^2)}\rightarrow A^*$ be an
isomorphism in $_{A^{(S^2)}}{\cal M}^H$. In particular $\theta$ is
an isomorphism of left $A$-modules, so $A$ is a Frobenius algebra.
By \cite{abrams1}, \cite{abrams2}, $A$ is a Frobenius algebra in
$k-mod$. Using the proof in the cited references, $A$ has a
coalgebra structure with comultiplication
$\delta=(\theta^{-1}\otimes \theta^{-1})\Delta^{cop}\theta$, where
$\Delta^{cop}$ is the comultiplication co-opposite to the
comultiplication $\Delta$ of $A^*$ induced by the multiplication
of $A$, and counit $\varepsilon_A=\theta (1_A)$. Moreover,
$\delta$ is a morphism of $A,A$-bimodules.

We show that $\delta$ and $\varepsilon$ are morphisms in ${\cal
M}^H$, and thus $A^{(S^2)}$ is a Frobenius algebra in the category
${\cal M}^H$. Since $\theta$ is an isomorphism in ${\cal M}^H$, in
order to show that $\delta$ is a morphism in ${\cal M}^H$, it is
enough to show that $\Delta^{cop}$ is so. As above, there exists
$\sum_i a_i\otimes h_i\otimes a_i^*\in A\otimes H\otimes A^*$ such
that $\sum a_0\otimes a_1=\sum_ia_i^*(a)a_i\otimes h_i$ for any
$a\in A$. Then $A^*\in {\cal M}^H$ with coaction $a^*\mapsto
\nu_1(a^*)=\sum a^*_0\otimes a^*_1=\sum_ia^*(a_i)a_i^*\otimes
S(h_i)$. Let $\nu_2:A^*\otimes A^*\rightarrow A^*\otimes
A^*\otimes H$ be the induced coaction. Then
$$(\Delta^{cop}\otimes I)\nu_1(a^*)=\sum_ia^*(a_i)(a_i^*)_2\otimes
(a_i^*)_1\otimes S(h_i)$$ and \bea (\nu_2\Delta^{cop})(a^*)&=&\sum
(a_2^*)_0\otimes (a_1^*)_0\otimes (a_2^*)_1(a_1^*)_1\\
&=&\sum_{i,j}a_2^*(a_j)a_1^*(a_i)a_j^*\otimes a_i^*\otimes
S(h_j)S(h_i)\\
&=&\sum_{i,j}a^*(a_ia_j)a_j^*\otimes a_i^*\otimes S(h_ih_j)\eea

If we use the natural isomorphism $A^*\otimes A^*\otimes H\simeq
Hom(A\otimes A,H)$, we see that showing that $(\Delta^{cop}\otimes
I)\nu_1(a^*)=(\nu_2\Delta^{cop})(a^*)$ is equivalent to showing
that for any $a,b\in A$ we have

\begin{equation} \label{eqinterm}
\sum_ia^*(a_i)(a_i^*)_2(a)(a_i^*)_1(b)S(h_i)=\sum_{i,j}a^*(a_ia_j)a_j^*(a)a_i^*(b)S(h_ih_j)
\end{equation}
Since $\sum (ba)_0\otimes (ba)_1=\sum b_0a_0\otimes b_1a_1$, we
have

\begin{equation} \label{eqcomalg}
\sum_i a_i^*(ba)a_i\otimes
h_i=\sum_{i,j}a_i^*(b)a_j^*(a)a_ia_j\otimes h_ih_j
\end{equation}

Now

\bea
\sum_{i,j}a^*(a_ia_j)a_j^*(a)a_i^*(b)S(h_ih_j)&=&\sum_ia_i^*(ba)a^*(a_i)S(h_i)
\;\;\;\mbox{by
}(\ref{eqcomalg}))\\&=&\sum_ia^*(a_i)(a_i^*)_1(b)(a_i^*)_2(a)S(h_i)\eea
which shows that (\ref{eqinterm}) holds.

Now since $\theta $ is $H$-colinear, we see that $$\sum_i
\theta(a)(a_i)a_i^*\otimes S(h_i)=\sum_i a_i^*(a)\theta
(a_i)\otimes S^2(h_i)$$ for any $a\in A$. This implies that
\begin{equation} \label{eqfinal}
\sum_i \theta(a)(a_i)a_i^*(1_A) S(h_i)=\sum_i a_i^*(a)\theta
(a_i)(1_A)S^2(h_i) \end{equation} Since $\sum_ia_i^*(1)a_i\otimes
h_i=1_A\otimes 1_H$, the left hand side in (\ref{eqfinal}) equals
$\theta(a)(1_A)S(1_H)=\theta(a)(1_A)1_H$. Since
$\theta(x)(y)=(x\theta(1_A))(y)=\theta(1_A)(yx)$, we obtain from
(\ref{eqfinal}) that
$$\theta (1_A)(a) 1_H=\sum_ia_i^*(a)\theta
(1_A)(a_i)S^2(h_i)$$ which means that $\varepsilon_A$ is
$H$-colinear.

$(5)\Rightarrow (3)$ Since $A^{(S^2)}$ is Frobenius in ${\cal
M}^H$, it has a coalgebra structure with comultiplication a
morphism of $A$-bimodules. Let
$\varepsilon_{A^{(S^2)}}:A^{(S^2)}\rightarrow k$ be the counit of
this coalgebra structure, which is a morphism in ${\cal M}^H$.
 Let ${\cal
H}=Ker(\varepsilon_{A^{(S^2)}})$, a hyperplane of $A$. By
\cite{abrams1}, \cite{abrams2}, $A$ is a Frobenius algebra with a
non-degenerate associative bilinear form $B:A\times A\rightarrow
k$, $B(b,a)=\varepsilon_{A^{(S^2)}}(ba)$. Then $\cal H$ does not
contain non-zero left ideals of $A$, since $Aa\subseteq {\cal H}$
implies that $B(A,a)=0$, so then $a=0$.  On the other hand
\bea \varepsilon_{A^{(S^2)}}((h^*S^2)\cdot a)&=&\varepsilon_{A^{(S^2)}}(\sum (h^*S^2)(a_1)a_0)\\
&=&(h^*S^2)(\sum \varepsilon_{A^{(S^2)}}(a_0)a_1)\\
&=&(h^*S^2)(\varepsilon_{A^{(S^2)}}(a)1)\\
&=&(h^*S^2)(1)\varepsilon_{A^{(S^2)}}(a)\\
&=&h^*(1)\varepsilon_{A^{(S^2)}}(a)\eea This implies that
$(h^*S^2)\cdot A\subseteq {\cal H}$ for any $h^*\in H^*$ with
$h^*(1)=0$.

\end{proof}

With arguments similar to the ones in the proof of Theorem
\ref{charFrobcomod}, one can prove the following. In fact, this
essentially follows from the theory of Frobenius algebras in
monoidal categories, see for instance \cite[Theorem 5.1]{bt}.
Indeed, any object $M$ of the category ${\cal M}^H$ has a left
dual $M^*$. Then a Frobenius algebra $A$ in the category ${\cal
M}^H$, which is known to have just $A$ as a left dual, must be
isomorphic to the left dual $A^*$, and this turns out to be even
an isomorphism of right $A$-modules.

\begin{theorem} \label{charFrobcomodright}
Let $A$ be a finite dimensional right $H$-comodule algebra.
The following assertions are equivalent.\\
{\rm (1)} $A$ is right $H$-Frobenius.\\
{\rm (2)} There exists a non-degenerate associative bilinear form
$B:A\times A\rightarrow k$ such that $B(h^*\cdot
a,b)=B(a,(h^*S)\cdot b)$ for any $a,b\in A$ and any $h^*\in H^*$.\\
{\rm (3)} $A$ has a hyperplane $\cal H$ which does not contain any
non-zero right ideal of $A$, and $h^*\cdot
A\subseteq {\cal H}$ for any $h^*\in H^*$ with $h^*(1)=0$.\\
{\rm (4)} $A$ has a hyperplane $\cal H$ which does not contain any
non-zero subobject of $A$ in ${\cal M}^H_A$, and $h^*\cdot
A\subseteq {\cal H}$ for any $h^*\in H^*$ with $h^*(1)=0$.\\
{\rm (5)} $A$ is a Frobenius algebra in the monoidal category
${\cal M}^H$.
\end{theorem}

The connection between the left-right $H$-Frobenius properties is
given by the following.

\begin{theorem}
Let $H$ be a Hopf algebra, and let $A$ be a finite dimensional
right $H$-comodule algebra. The following assertions hold.\\
{\rm (1)} If $A$ is right $H$-Frobenius, then $A$ is left $H$-Frobenius.\\
{\rm (2)} If the antipode $S$ is injective and $A$ is left
$H$-Frobenius, then $A$ is also right $H$-Frobenius.
\end{theorem}
\begin{proof}
Let $\rho:A\rightarrow A\otimes H$ be the $H$-comodule structure
of $A$, and denote by $\rho_{A\otimes A}$ the $H$-comodule
structure on $A\otimes A$. Then the right $H$-comodule structures
of $A^{(S^2)}$, respectively $A^{(S^2)}\otimes A^{(S^2)}$, are
given by $(I\otimes S^2)\rho$, respectively $(I\otimes I\otimes
S^2)\rho_{A\otimes A}$. Let $\delta:A\rightarrow A\otimes A$ be a
linear map. Consider the following diagram.

\[
\xymatrix{
A \ar[r]^-\rho \ar[d]_{\delta} & A\otimes H \ar[d]^{\delta\otimes I}\ar[r]^-{I\otimes S^2}&A\otimes H \ar[d]^{\delta\otimes I}\\
A\otimes A \ar[r]^-{\rho_{A\otimes A}} & A\otimes A\otimes H
\ar[r]^-{I\otimes I\otimes S^2} & A\otimes A\otimes H}
\]

If $A$ is right Frobenius, one can choose such a $\delta$ which is
coassociative, a morphism of $A$-bimodules and a morphism of right
$H$-comodules. Then the left square of the diagram is commutative,
therefore, since the right square is always commutative, we get
that the big rectangle is commutative, showing that $\delta$ is
also a morphism of right $H$-comodules when regarded as
$\delta:A^{(S^2)}\rightarrow A^{(S^2)}\otimes A^{(S^2)}$. A
similar argument works for the counit, and we conclude that if $A$
has a coalgebra structure in ${\cal M}^H$ such that the
comultiplication is a morphism of $A$-modules, then $A^{(S^2)}$
has the same property. Thus $(1)$ holds. For $(2)$, we see that if
the big rectangle commutes, then so does the left square, since
$S^2$ is injective, and the rest of the argument is as in $(1)$.
\end{proof}

\section{$\sigma$-graded Frobenius algebras}
\label{ssigmaFrobenius}

Let $A=\bigoplus\limits_{g \in G}A_g$ be a $G$-graded algebra,
i.e. a $k$-algebra which is a direct sum of the subspaces $A_g$,
such that $A_gA_h\subseteq A_{gh}$ for any $g,h\in G$. This is
equivalent to $A$ being an algebra in ${\cal M}^{kG}$, the
monoidal category of $G$-graded vector spaces. The category
$_A{\cal M}^{kG}$ of left Doi-Hopf modules is just the category of
graded left $A$-modules, and we use the standard notation $A-gr$
for it. The category of graded right $A$-modules is denoted by
$gr-A$. If $M=\bigoplus\limits_{g \in G}M_g$ is a graded left
$A$-module of finite support, then the dual space $M^*$ is a
graded right $A$-module, with the grading $M^*=\bigoplus\limits_{g
\in G}(M^*)_g$, where $(M^*)_g=\{ f\in M^*|\; f(M_h)=0\;\mbox{for
any }h\neq g^{-1}\}$. In a similar way, the dual space of a graded
right $A$-module of finite support is a graded left $A$-module. In
particular, if $A$ has finite support, then $A^*$ is a graded left
$A$-module and a graded right $A$-module, i.e. it is a left $A$,
right $A$-graded bimodule.

If $M=\bigoplus\limits_{g \in G}M_g$ is a graded left $A$-module
and $\sigma \in G$, then the $\sigma$-suspension of $M$ is the
graded left $A$-module $M(\sigma)$, which is just $M$ as an
$A$-module, and has the grading defined by
$M(\sigma)_g=M_{g\sigma}$ for any $g\in G$.

If $M=\bigoplus\limits_{g \in G}M_g$ is a graded right $A$-module
and $\sigma \in G$, the $\sigma$-suspension of $M$, denoted by
$(\sigma)M$, is the $A$-module $M$ with the grading defined by
$(\sigma)M_g=M_{\sigma g}$.

\begin{lemma}\label{lemadualsuspensie}
1) If $M\in A-gr$ has finite support and $\sigma \in G$, then
$M(\sigma)^*= (\sigma^{-1})M^*$ in $gr-A$.\\
2) If $M\in gr-A$ has finite support and $\sigma \in G$, then
$(\sigma)M^*= M^*(\sigma^{-1})$ in $A-gr$.
\end{lemma}
\begin{proof}
Let $g\in G$. We have that \bea (M(\sigma)^*)_g&=&\{ f\in
M^*|f(M(\sigma)_{\lambda})=0 \mbox{ for any }\lambda\neq
g^{-1}\}\\
&=&\{ f\in M^*|f(M_{\lambda\sigma})=0 \mbox{ for any }\lambda\neq
g^{-1}\}\\
&=&\{ f\in M^*|f(M_{\tau})=0 \mbox{ for any }\tau\neq
g^{-1}\sigma\}\\
&=&(M^*)_{\sigma^{-1}g}\\
&=&((\sigma^{-1})M^*)_g \eea and (1) follows. (2) is similar.
\end{proof}

\begin{lemma} \label{lemadubludual}
If $M$ is a finite dimensional left (respectively right) graded
$A$-module, then $(M^*)^*\simeq M$ in $A-gr$ (respectively
$gr-A$).
\end{lemma}
\begin{proof} It is easy to check that the natural isomorphism
$M\simeq M^{**}$ of $A$-modules is in fact a graded isomorphism.
\end{proof}

\begin{proposition}\label{defgrFrob}
Let $A$ be a finite dimensional $G$-graded algebra, and let
$\sigma \in G$. The following conditions are equivalent.\\
(i) $A(\sigma)\simeq A^*$ in $A-gr$.\\
(ii) $(\sigma)A\simeq A^*$ in $gr-A$.
\end{proposition}
\begin{proof}
(i)$\Rightarrow $(ii) Since $A(\sigma)\simeq A^*$ in $A-gr$, by
using Lemma \ref{lemadualsuspensie} and Lemma \ref{lemadubludual},
we have that $A\simeq A^{**}\simeq (A(\sigma))^*\simeq
(\sigma^{-1})A^*$ in $gr-A$.  Therefore $A^*\simeq (\sigma)A$ in $gr-A$. \\
(ii)$\Rightarrow$ (i) is similar.
\end{proof}

\begin{definition}
A finite dimensional $G$-graded algebra $A$ is called
$\sigma$-graded Frobenius if it satisfies the equivalent
conditions in Proposition \ref{defgrFrob}. Clearly, $e$-graded
Frobenius means just (left, or equivalently right) $kG$-Frobenius,
and we simply say in this case that $A$ is graded Frobenius.
\end{definition}

\begin{remark} \label{remarcagrFrobenius}
(1) If $A$ is $\sigma$-graded Frobenius for some $\sigma$, then
obviously $A$ is Frobenius as an algebra.

 (2) If $A$ is
$\sigma$-graded Frobenius, and $H$ is a subgroup of $G$ such that
$\sigma\in H$, then $A_H=\bigoplus\limits_{g\in H}A_g$ is
$\sigma$-graded Frobenius as an $H$-graded algebra. Indeed, this
follows from the fact that an isomorphism $A(\sigma)\simeq A^*$ of
$G$-graded left $A$-modules, induces, by restriction to the
components of degrees in $H$, an isomorphism of $H$-graded left
$A_H$-modules $A_H(\sigma)\simeq A_H^*$. In particular, if $A$ is
graded Frobenius, then $A_e$ is a Frobenius algebra.

 (3) If $A$ is
$\sigma$-graded Frobenius, then for $\tau\in G$ we have that $A$
is $\tau$-graded Frobenius if and only if $\tau\in \sigma G\{
A\}$, where $G\{ A\}=\{ g\in G\;|\; A(g)\simeq A \mbox{ in
}A-gr\}$ is the inertia group of $A$.

 (4) If
$A=\bigoplus\limits_{g\in G}A_g$ is a $G$-graded algebra and $H$
is a normal subgroup of $G$, then $A$ also has a $G/H$-grading
given by $A_{\hat{\sigma}}=\bigoplus\limits_{g\in
\hat{\sigma}}A_g$ for any $\hat{\sigma}\in G/H$. If $A$ is graded
Frobenius with the $G$-grading, then it is also graded Frobenius
with the $G/H$-grading.
\end{remark}

\section{Characterizing $\sigma$-graded Frobenius} \label{scaract}

The characterization of the Frobenius property in the ungraded
case (see for example \cite[Theorem 3.15]{lam}) carries on to the
graded case. The proof is on the same line. For $\sigma=e$, the
result is just a particular case of Theorem \ref{charFrobcomod}.

\begin{theorem}\label{caractgrFrobenius}
Let $A=\bigoplus\limits_{g \in G}A_g$ be a finite dimensional
$G$-graded algebra, and let $\sigma \in G$. The following assertions are equivalent.\\
(i) $A$ is $\sigma$-graded Frobenius.\\
(ii) There exists a non-degenerate associative bilinear form
$B:A\times A\rightarrow k$ such that $B(r_\tau,r_\mu)=0$ for any
$r_\tau\in A_\tau$ and $r_\mu\in A_\mu$ with $\tau\mu\neq
\sigma$.\\
(iii) There exists a hyperplane $\cal H$ in $A$ such that
$A_\tau\subseteq {\cal H}$ for any $\tau\neq \sigma$, and $\cal H$
does not contain nonzero graded left ideals.
\end{theorem}
\begin{proof}
(i)$\Rightarrow$(ii) Since $A(\sigma)\simeq A^*$, we have that
$A\simeq A^*(\sigma^{-1})$ in $A-gr$. Let $\phi:A\rightarrow
A^*(\sigma)$ be an isomorphism in $A-gr$. Then $B:A\times
A\rightarrow k$, $B(x,y)=\phi(y)(x)$ is associative and
non-degenerate. Let $r_\tau\in A_\tau$ and $r_\mu\in A_\mu$ with
$\tau\mu\neq \sigma$. Then $\phi(r_\mu)\in
A^*(\sigma^{-1})_\mu=A^*_{\mu\sigma^{-1}}$, so
$\phi(r_\mu)(A_\lambda)=0$ for any $\lambda\neq
(\mu\sigma^{-1})^{-1}=\sigma\mu^{-1}$. Since $\tau\neq
\sigma\mu^{-1}$, we have that
$B(r_\tau,r_\mu)=\phi(r_\mu)(r_\tau)=0$.

 (ii)$\Rightarrow$ (iii)
As in the ungraded case, the set ${\cal H}=\{ r\in A\;|\;
B(1,r)=0\}$ is a hyperplane which does not contain nonzero left
ideals of $A$, so then neither nonzero graded left ideals. Since
$B(1,r_\tau)=0$ for any $\tau\neq \sigma$, we see that
$A_\tau\subseteq {\cal H}$ for any $\tau\neq \sigma$.

 (iii)$\Rightarrow (i)$ Let $\pi:A\rightarrow
A$ be a linear map with $Ker(\pi)={\cal H}$. Define
$\phi:A(\sigma)\rightarrow A^*$ by $\phi(y)(x)=\pi(xy)$ for any
$x,y\in A$, which is a morphism of left $A$-modules. We show that
$\phi$ is a morphism in $A-gr$. Indeed, let $r_{\tau\sigma}\in
A_{\tau\sigma}=A(\sigma)_{\tau}$. Then for $\lambda\neq \tau^{-1}$
we have that $\lambda\tau\sigma\neq \sigma$, so
$A_{\lambda\tau\sigma}\subseteq H$. Then
$\phi(r_{\tau\sigma})(A_\lambda)=\pi (A_\lambda
r_{\tau\sigma})\subseteq \pi(A_{\lambda\tau\sigma})=0$, thus
$\phi(r_{\tau\sigma})\in A^*_\tau$.

Finally, $\phi$ is injective, and hence an isomorphism. Indeed, if
the graded left ideal $Ker(\phi)$ is nonzero, let $y\in Ker(\phi)$
be a nonzero homogeneous element. Since $\pi (xy)=\phi(y)(x)=0$
for any $x\in A$, we have that $Ay\subseteq Ker(\pi)={\cal H}$, a
contradiction.
\end{proof}

\begin{corollary} \label{Frobeniusstronglygraded}
Let $A$ be a strongly graded finite dimensional $G$-graded
algebra. Then $A$ is graded Frobenius if and only if $A_e$ is
Frobenius.
\end{corollary}
\begin{proof}
If $A$ is graded Frobenius, then $A_e$ is Frobenius by Remark
\ref{remarcagrFrobenius}(2). Conversely, assume that $A_e$ is a
Frobenius algebra, and let $\cal T$ be a hyperplane in $A_e$ which
does not contain nonzero left ideals. Then ${\cal H}={\cal
T}\oplus (\bigoplus\limits_{g\in G\setminus \{e\}}A_g)$ is a
hyperplane in $A$, and clearly $A_g\subseteq {\cal H}$ for any
$g\neq e$. If $I$ is a graded left ideal of $A$ contained in $\cal
H$, then $I_e\subseteq {\cal T}$. Since $I_e$ is a left ideal in
$A_e$, it must be zero. Then $I$ is also zero, since $A$ is
strongly graded.
\end{proof}

Let $A=\bigoplus\limits_{g \in G}A_g$ be a finite dimensional
$G$-graded algebra. Let $\sigma \in G$, and let ${\cal C}_\sigma$
be the localizing subcategory of $A-gr$ consisting of all objects
$M$ such that $M_\sigma=0$. Let $t_{{\cal C}_\sigma}$ be the
associated radical, i.e. for a graded $A$-module $M$, $t_{{\cal
C}_\sigma}(M)$ is the largest graded submodule of $M$ whose
homogeneous component of degree $\sigma$ is zero. A graded left
$A$-module $M$ is called $\sigma$-faithful if $t_{{\cal
C}_\sigma}(M)=0$. The graded ring $A$ is called left
$\sigma$-faithful if it is $\sigma$-faithful as a left graded
$A$-module, i.e. for any non-zero $a_g\in A_g$, $g\in G$, we have
that $A_{\sigma g^{-1}}a_g\neq 0$. Similarly, $A$ is called right
$\sigma$-faithful if for any non-zero $a_g\in A_g$, $g\in G$, we
have that $a_gA_{g^{-1}\sigma}\neq 0$.

We recall from \cite[Section 2.5]{nvo} that the coinduced functor
$Coind:A_e-mod\rightarrow A-gr$ associates to a left $A_e$-module
$N$ the left $A$-module $Coind(N)=Hom_{A_e}(A,N)$ (with $A$-module
structure induced by the right $A$-module structure of $A$), with
the grading such that the homogeneous component of degree $g$
consists of all the maps $f\in Hom_{A_e}(A,N)$ vanishing on $A_h$
for any $h\neq g^{-1}$. A right adjoint of the functor
$(-)_{\sigma}:A-gr\rightarrow A_e-mod$, which associates to a
graded module its homogeneous component of degree $\sigma$, is
$T_{\sigma^{-1}} \circ Coind$, where
$T_{\sigma^{-1}}:A-gr\rightarrow A-gr$ is the isomorphism of
categories taking a graded module $M$ to its
$\sigma^{-1}$-suspension $M(\sigma^{-1})$. The unit of this
adjunction is defined as follows. For $M\in A-gr$, let
$\nu_M:M\rightarrow Coind(M_{\sigma})(\sigma^{-1})$,
$\nu_M(x_\lambda)(a)=a_{\sigma\lambda^{-1}}x_\lambda$ for any
$\lambda\in G$, $x_\lambda\in M_\lambda$, $a\in A$. Then $\nu_M$
is a morphism in $A-gr$, moreover $Ker(\nu_M)=t_{{\cal
C}_{\sigma}}(M)$. Now we have the following characterization of
the $\sigma$-graded Frobenius property.

\begin{theorem} \label{thcaractgrFrob}
Let $A=\bigoplus\limits_{g \in G}A_g$ be a finite dimensional
$G$-graded algebra, and let $\sigma \in G$. The following assertions are equivalent.\\
(i) $A$ is $\sigma$-graded Frobenius.\\
(ii) $A_{\sigma}\simeq A_e^*$ as left $A_e$-modules, and $A$
is left $\sigma$-faithful.\\
(iii) $A_{\sigma}\simeq A_e^*$ as right $A_e$-modules, and $A$ is
right $\sigma$-faithful.
\end{theorem}
\begin{proof}
(i)$\Rightarrow$(ii) An isomorphism $A(\sigma)\simeq A^*$ of
graded left $A$-modules induces an isomorphism of left
$A_e$-modules between the homogeneous components of degree $e$,
i.e. $A_{\sigma}\simeq A_e^*$.

On the other hand, $A^*$ is $e$-faithful. Indeed, if $f\in A^*_g$,
$g\in G$, and $A_{g^{-1}}f=0$, then for any $r_{g^{-1}}\in
A_{g^{-1}}$ we have
$f(r_{g^{-1}})=f(1r_{g^{-1}})=(r_{g^{-1}}f)(1)=0$. But $f(A_h)=0$
for any $h\neq g^{-1}$, so $f$ must be zero. Now $A(\sigma)$ is
$e$-faithful, which implies that $A$ is $\sigma$-faithful.

 (ii)$\Rightarrow$(i) Let
$\theta:A_{\sigma}\rightarrow A_e^*$ be an isomorphism of left
$A_e$-modules, and let
$$\overline{\theta}:Hom_{A_e}(A,A_{\sigma})\rightarrow
Hom_{A_e}(A,A_e^*), \overline{\theta}(f)=\theta f$$ be the induced
isomorphism of left $A$-modules. Now let
$$\phi:Hom_{A_e}(A,A_e^*)\rightarrow Hom_k(A_e\otimes _{A_e}A,k)$$
defined by $$ \phi(F)(a\otimes r)=F(r)(a) \; \mbox{for } F\in
Hom_{A_e}(A,A_e^*), r\in A, a\in A_e,$$ be the natural isomorphism
of left $A$-modules. Finally let $\gamma:A\rightarrow A_e\otimes
_{A_e}A$ be the natural isomorphism, and
$$\overline{\gamma}:Hom_k(A_e\otimes _{A_e}A,k)\rightarrow
Hom_k(A,k), \overline{\gamma}(f)=f\gamma$$ be the induced
isomorphism of left $A$-modules.

Therefore $\overline{\gamma}\phi
\overline{\theta}:Coind(A_{\sigma})\rightarrow A^*$ is an
isomorphism of left $A$-modules. We show that it is also a
morphism in $A-gr$. Indeed, let $u\in Coind(A_{\sigma})_g$, so
$u:A\rightarrow A_{\sigma}$ and $u(A_h)=0$ for any $h\neq g^{-1}$.
Then for $a_h\in A_h$, with $h\neq g^{-1}$, we have  \bea
((\overline{\gamma}\phi
\overline{\theta})(u))(a_h)&=&(\overline{\gamma}\phi
(\theta u))(a_h)\\
&=&(\phi (\theta u))\gamma)(a_h)\\
&=&\phi(\theta u)(1\otimes a_h)\\
&=&(\theta u)(a_h)(1)\\
&=&\theta (u(a_h))(1)\\
&=&0 \eea

By applying the isomorphism of categories $T_{\sigma^{-1}}$, we
can regard as $\overline{\gamma}\phi
\overline{\theta}:Coind(A_{\sigma})(\sigma^{-1})\rightarrow
A^*(\sigma^{-1})$, a graded isomorphism. On the other hand
$\nu_A:A\rightarrow Coind(A_{\sigma})(\sigma^{-1})$ is an
injective morphism of graded $A$-modules, since its kernel is
$t_{{\cal C}_\sigma}(A)$ and $A$ is left $\sigma$-faithful. Hence
$\overline{\gamma}\phi \overline{\theta}\nu_A:A\rightarrow
A^*(\sigma^{-1})$ is an injective morphism in $A-gr$. Since $A$
and $A^*$ have the same dimension, this injective morphism must be
bijective, so then $A\simeq A^*(\sigma^{-1})$, and
$A(\sigma)\simeq A^*$.

(i)$\Leftrightarrow$ (iii) is similar to (i)$\Leftrightarrow$
(ii).
\end{proof}

In particular, we obtain the description of Frobenius algebras in
the category of $G$-graded vector spaces.

\begin{corollary}
Let $A$ be a graded algebra. Then $A$ is graded Frobenius if and
only if $A_e$ is Frobenius and $A$ is left (or right)
$e$-faithful.
\end{corollary}

We recall that a graded algebra $A$ is called graded semisimple if
$A$ is a direct sum of minimal graded left ideals. This is
equivalent to  $A-gr$ being a semisimple category.

\begin{corollary}
A graded semisimple algebra is graded Frobenius. In particular, a
graded division algebra is graded Frobenius.
\end{corollary}
\begin{proof}
If $A$ is graded semisimple, then it is easy to see that $A_e$ is
a semisimple algebra, in particular it is Frobenius. On the other
hand, by \cite[Proposition 2.9.6]{nvo}, if $A$ is graded
semisimple, the $A$ is $e$-faithful. Now $A$ is graded Frobenius
by Theorem \ref{thcaractgrFrob}.
\end{proof}

In Section \ref{sectionexamples} we will give examples graded
algebras which are Frobenius, but not graded Frobenius. However,
we have the following.

\begin{theorem} \label{teoremalocal}
Let $A=\bigoplus\limits_{g \in G}A_g$ be a finite dimensional
graded algebra such that $A$ is a Frobenius algebra and $A_e$ is a
local ring. Then there exists $\sigma \in G$ such that $A$ is
$\sigma$-graded Frobenius.
\end{theorem}
\begin{proof}
The graded left $A$-module $A$ is indecomposable. Indeed, if
$A=X\oplus Y$ in $A-gr$, then $X=Au$ and $Y=Av$ for some
orthogonal idempotents $u,v$ in $A_e$ such that $u+v=1$. Since
$A_e$ is local, we have $u=0$ or $v=0$, so $X=0$ or $Y=0$. We also
have that $A^*$ is an indecomposable graded left module. Indeed,
$A^*=X\oplus Y$ in $A-gr$ implies that $A^{**}=X^*\oplus Y^*$ in
$gr-A$. Since $A^{**}\simeq A$ in $gr-A$, and $A$ is
indecomposable in $gr-A$ (the argument above works also on the
right), we see that $X^*=0$ or $Y^*=0$, therefore $X=0$ or $Y=0$.

Since $A$ is Frobenius, the left $A$-modules $A$ and $A^*$ are
injective, hence they are also injective objects in $A-gr$, by
\cite[Corollary 2.3.2]{nvo}. Now we consider the forgetful functor
$U:A-gr\rightarrow A-mod$, which has a right adjoint $F$ (see
\cite[Theorem 2.5.1]{nvo}). It is known (see \cite[Proposition
2.5.4]{nvo}) that for any $M\in A-gr$, we have that $F(U(M))\simeq
\bigoplus\limits_{g \in G}M(g)$. Then if we apply $F$ to the
isomorphism of left $A$-modules $A\simeq A^*$, we obtain an
isomorphism $\bigoplus\limits_{g \in G}A(g)\simeq
\bigoplus\limits_{g \in G}A^*(g)$ in $A-gr$. In both sides we have
direct sums of injective indecomposable objects in $A-gr$, so by
Azumaya Theorem, we have that the decompositions are equivalent,
in particular there is $\sigma \in G$ such that $A^*\simeq
A(\sigma)$.
\end{proof}

If we drop the condition that $A_e$ is local, the above result
does not hold anymore, see Example \ref{exemplulocal}. As a
biproduct of the method in the proof of Theorem \ref{teoremalocal}
we obtain the following.

\begin{proposition}
Let $A$ be an algebra graded by a finite group $G$, and let $M$
and $N$ be finite dimensional graded left $A$-modules such that
$M$ is graded indecomposable. If $M\simeq N$ as $A$-modules, then
$M\simeq N(\sigma)$ as graded $A$-modules for some $\sigma \in G$.
\end{proposition}
\begin{proof}
If $F$ is the right adjoint of the forgetful functor
$U:A-gr\rightarrow A-mod$, then we have $F(M)\simeq F(N)$, so then
$\bigoplus\limits_{g \in G}M(g)\simeq \bigoplus\limits_{g \in
G}M(g)$ in $A-gr$. Since each $M(g)$ is indecomposable in $A-gr$
and the two sides have the same finite number of summands, the
Krull-Schmidt Theorem shows that $N$ must be graded indecomposable
and $M\simeq N(\sigma)$ for some $\sigma$.
\end{proof}

\section{Graded symmetric algebras} \label{ssimetric}

We consider a special class of graded Frobenius algebras. These
are just the symmetric algebras in the sovereign monoidal category
of $G$-graded vector spaces, see \cite{fuchs}.

\begin{definition}
A finite dimensional graded algebra $A$ is called graded symmetric
if $A$ and $A^*$ are isomorphic as graded left-$A$, right-$A$
bimodules.
\end{definition}

Obviously, if $A$ is graded symmetric, then $A$ is symmetric as an
algebra, and also $A$ is graded Frobenius. If $A$ is graded
symmetric and $H$ is a subgroup of the grading group $G$, then
$A_H$ is a graded symmetric algebra of type $H$. In particular
$A_e$ is a symmetric algebra whenever $A$ is graded symmetric. The
following characterizes the graded symmetric property. The proof
is adapted from the ungraded case (see \cite[Theorem 16.54]{lam}).

\begin{theorem} \label{theoremsymmetric}
Let $A=\bigoplus\limits_{g \in G}A_g$ be a finite dimensional
$G$-graded algebra. The following assertions are equivalent.\\
(i) $A$ is graded symmetric.\\
(ii) There exists a non-degenerate associative symmetric bilinear
form $B:A\times A\rightarrow k$ such that $B(r_\tau,r_\mu)=0$ for
any $r_\tau\in A_\tau$ and $r_\mu\in A_\mu$ with $\tau\mu\neq
e$.\\
(iii) There exists a linear map $\lambda:A\rightarrow k$ such that
$\lambda(xy)=\lambda(yx)$ for any $x,y\in A$, $Ker(\lambda)$ does
not contain non-zero graded left ideals, and $\lambda(x_\sigma
y_\tau)=0$ for any $x_\sigma\in A_\sigma,y_\tau\in A_\tau$ with
$\sigma\tau\neq e$.
\end{theorem}
\begin{proof}
(i)$\Rightarrow$(ii) Let $\phi:A\rightarrow A^*$ be an isomorphism
of graded bimodules. As in the ungraded case, $B:A\times
A\rightarrow k$, $B(x,y)=\phi(y)(x)$ is a non-degenerate
associative symmetric bilinear form. As in the proof of Theorem
\ref{caractgrFrobenius}, (i)$\Rightarrow$ (ii), we see that
$B(r_\tau,r_\mu)=0$ for any $r_\tau\in A_\tau$ and $r_\mu\in
A_\mu$ with $\tau\mu\neq e$.

(ii)$\Rightarrow$(iii) As in the ungraded case, the linear map
$\lambda:A\rightarrow k$, $\lambda(x)=B(x,1)$ satisfies
$\lambda(xy)=\lambda(yx)$ for any $x,y\in A$, and moreover,
$Ker(\lambda)$ does not even contain non-zero left ideals. If
$x_\sigma\in A_\sigma,y_\tau\in A_\tau$ with $\sigma\tau\neq e$,
then $\lambda (x_\sigma y_\tau)=B(x_\sigma y_\tau,1)=B(x_\sigma,
y_\tau)=0$.

(iii)$\Rightarrow$ (i) Let $\phi:A\rightarrow A^*$, $\phi
(y)(x)=\lambda (xy)$. Then $\phi$ is an isomorphism of left $A$,
right $A$-bimodules, as in the ungraded case. If $y_\tau \in
A_\tau$, then $\phi (y_\tau)(x_\sigma)=0$ for any $x_\sigma\in
A_\sigma$ with $\sigma \neq \tau^{-1}$. This shows that
$\phi(y_\tau)\in A^*_\tau$, so $\phi$ is graded.
\end{proof}

\begin{remark}
Let $A$ be a strongly graded algebra. We showed in Corollary
\ref{Frobeniusstronglygraded} that $A$ is graded Frobenius if and
only if $A_e$ is Frobenius. For the symmetric property, this
equivalence does not hold anymore. If $A$ is graded symmetric, we
have seen that $A_e$ must be symmetric. We give an example showing
that the converse is not true. Assume that the base field $k$ has
characteristic $\neq 2$, and let $R$ be a symmetric algebra and an
algebra automorphism $\alpha$ of order 2 of $R$ such that the
invariant subalgebra $R^\alpha$ is not symmetric (see
\cite[Exercise 32, page 457]{lam}). Then the cyclic group $C_2=\{
e,\alpha\}$ acts as automorphisms on $R$, and let $A=R*C_2$ be the
associated skew group algebra. We have that
$u=\frac{1}{2}(e+\alpha)$ is an idempotent of $A$ and $uAu\simeq
R^\alpha$, which is not symmetric. By \cite[Exercise 25, page
456]{lam}, $A$ cannot be symmetric, so then it is not graded
symmetric when regarded as a $C_2$-graded algebra. However
$A_e\simeq R$ is symmetric.

This example also shows that there is not a version of Theorem
\ref{thcaractgrFrob} for the graded symmetric property. It is
clear that if $A$ is a finite dimensional $G$-graded algebra which
is graded symmetric, then $A_e$ is symmetric and $A$ is
$e$-faithful (to the left and to the right). The $e$-faithfulness
follows from the fact that $A$ is graded Frobenius, using Theorem
\ref{thcaractgrFrob}. However, the converse is not true. It is
possible that $A_e$ is symmetric and $A$ is (left and right)
$e$-faithful, but $A$ is not graded symmetric. Indeed, let $A$ be
the $C_2$-graded algebra constructed in the first part of this
example. Then $A_e=R$ is symmetric, and $A$ is left and right
$e$-faithful, since it is strongly graded. However, $A$ is not
graded symmetric (in fact not even symmetric).
\end{remark}

Now we discuss tensor products of graded Frobenius (symmetric)
algebras. Assume that $A=\bigoplus\limits_{g \in G}A_g$ and
$B=\bigoplus\limits_{g \in G}B_g$ are finite dimensional
$G$-graded algebras such that $\sigma \tau=\tau\sigma$ for any
$\sigma\in supp(A)$ and $\tau\in supp (B)$. Then the tensor
product of algebras $A\otimes B$ is $G$-graded, with the grading
given by $(A\otimes
B)_\sigma=\bigoplus\limits_{gh=\sigma}A_g\otimes B_h$. Under these
conditions we have the following.

\begin{proposition}
If $A$ and $B$ are graded Frobenius (respectively graded
symmetric), then so is $A\otimes B$.
\end{proposition}
\begin{proof}
Let $\phi:A\rightarrow A^*$ be an isomorphism of graded left
$A$-modules (respectively bimodules), and let $\psi:B\rightarrow
B^*$ be an isomorphism of graded left $B$-modules (respectively
bimodules). Let $F=\gamma (\phi\otimes \psi)$, where
$\gamma:A^*\otimes B^*\rightarrow (A\otimes B)^*$ is the natural
linear isomorphism. Then it is easy to check that $F$ is an
isomorphism of left $A\otimes B$-modules (respectively bimodules),
and that $F$ is moreover graded.
\end{proof}

It is known that a division algebra is symmetric, see
\cite[Example 16.59]{lam1}.  At this point we do not know whether
any graded division algebra is graded symmetric. The following
gives a partial answer.

\begin{theorem} \label{grdivringsim}
Let $\Delta=\bigoplus\limits_{g \in G}\Delta_g$ be a graded
division $k$-algebra such that $Z(\Delta_e)=Z(\Delta)\cap
\Delta_e$. Then $\Delta$ is graded symmetric.
\end{theorem}
\begin{proof}
We claim that the $k$-subspace of $\Delta_e$ spanned by all the
commutators of the form $xy-yx$ with $x\in \Delta_g$, $y\in
\Delta_{g^{-1}}$, $g\in G$, is not the whole of $\Delta_e$. Then
we can choose a hyperplane $H$ of $\Delta$ containing all these
commutators and all homogeneous components $\Delta_g$ with $g\neq
e$. Noting that the only graded left ideals of $\Delta$ are 0 and
$\Delta$, we have that a linear map $\lambda:\Delta\rightarrow k$
with $Ker(\lambda)=H$ satisfies the conditions in Theorem
\ref{theoremsymmetric}(3), thus making $\Delta$ a graded symmetric
algebra. Let $F=Z(\Delta_e)\subseteq Z(\Delta)$. Obviously
$k\subseteq F$. The claim above obviously follows if we show that
the $F$-subspace $S$ of $\Delta_e$ spanned by all commutators as
above is not the whole of $\Delta_e$.

Assume otherwise that $S=\Delta_e$. If $x\in \Delta_g$ and $y\in
\Delta_{g^{-1}}$ with $g\neq e$, and $x,y\neq 0$, then
$xy-yx=xyxx^{-1}-yx=xux^{-1}-u$, where $u=yx\in \Delta_e$. Thus
each such commutator is of the form $\phi(u)-u$, where $\phi$ is
an $F$-automorphism of $\Delta_e$ (note that the inner
automorphism of $\Delta_e$ associated to some $x\in
\Delta_g\setminus\{0\}$ is an $F$-automorphism since $F\subseteq
Z(\Delta)$). Therefore $\Delta_e$ is the $F$-span of some
commutators $xy-yx$ with $x,y\in \Delta_e$, and some elements of
the form $\phi(u)-u$, where $\phi$ is an $F$-automorphism of
$\Delta_e$, and $u\in \Delta_e$. Let $\overline{F}$ be the
algebraic closure of $F$. Lifting to $\overline{F}$, we obtain
that $\overline{F}\otimes _F\Delta_e$ is $\overline{F}$-spanned by
some elements of the form $ab-ba$, and some elements of the form
$\psi(a)-a$, where $a,b\in \overline{F}\otimes _F\Delta_e$ and
$\psi$ is an automorphism of the $\overline{F}$-algebra
$\overline{F}\otimes _F\Delta_e$. As a consequence of
\cite[Theorem 15.1]{lam1} we have that $\overline{F}\otimes
_F\Delta_e$ is isomorphic to a matrix algebra $M_r(\overline{F})$
as an $\overline{F}$-algebra. But if $a,b\in M_r(\overline{F})$,
then $ab-ba$ has trace 0, and since an automorphism of
$M_r(\overline{F})$ is inner, an element of the form $\psi(a)-a$,
where $a\in M_r(\overline{F})$ and $\psi$ is an automorphism of
the $\overline{F}$-algebra $M_r(\overline{F})$, has also trace 0.
This is a contradiction, since a family of elements of trace 0 of
$M_r(\overline{F})$ cannot span $M_r(\overline{F})$, and this ends
the proof.
\end{proof}

 The following shows that any graded division
algebra has a distinguished graded sub-division algebra which is
graded symmetric.

\begin{corollary}
Let $A$ be a $G$-graded division algebra. Then $\Delta=C_A(A_e)$,
the centralizer of $A_e$ in $A$, is a graded division algebra
which is graded symmetric.
\end{corollary}
\begin{proof}
It is clear that $\Delta$ is a graded division algebra. If we take
into account Theorem \ref{grdivringsim}, it is enough to show that
$Z(\Delta_e)=Z(\Delta)\cap \Delta_e$. To see this, if we take
$a\in Z(\Delta_e)=\Delta_e$, then $ay=ya$ for any $y\in \Delta$,
so $a\in Z(\Delta)$. Thus $Z(\Delta_e)\subseteq Z(\Delta)\cap
\Delta_e$. The converse inclusion always holds.
\end{proof}

\section{Examples} \label{sectionexamples}

In this section we give examples to illustrate the concepts of
$\sigma$-graded Frobenius algebra and graded symmetric algebra,
and their connection to the Frobenius and the symmetric
properties.

\begin{example} \label{extrivext}
Let $R$ be a finitely generated $k$-algebra, and $R^*$ the dual
space of $R$, with the usual bimodule structure. The trivial
extension $A=R\oplus R^*$ is the algebra with multiplication
defined by $(r,f)(r',f')=(rr', rf'+fr')$. It is known (see
\cite[Example 16.60]{lam}) that $A$ is a symmetric algebra. On the
other hand, $A$ is a ${\bf Z}_2$-graded algebra, with homogeneous
components
$A_{\hat{0}}=R\oplus 0\simeq R$, $A_{\hat{1}}=0\oplus R^*$.\\
We have that $A$ is not $\hat{0}$-faithful, since
$A_{\hat{1}}(0,f)=0$ for any $f\in R^*$. On the other hand, $A$ is
 $\hat{1}$-faithful. Indeed, if $A_{\hat{0}}(0,f)=0$, then $Rf=0$,
 so $f=0$; also, if $A_{\hat{1}}(r,0)=0$, then $R^*r=0$, so
 $f(r)=(fr)(1)=0$ for any $f\in R^*$, showing that $r=0$. \\
 We conclude that $A$ is Frobenius (since it is symmetric), but it
 is not graded Frobenius (since it is not $\hat{0}$-faithful).
 However, $A$ is $\hat{1}$-graded Frobenius. This also gives an
 example of a graded algebra which is symmetric, but not graded
 symmetric.
\end{example}

\begin{example}
Let $u$ and $v$ be non-zero elements of $k$, and let $A$ be the
set of all matrices of the form ${\tiny \left(
\begin{array}{cccc}
a&b&c&d\\
0&a&0&uc\\
0&0&a&vb\\
0&0&0&a
\end{array}
\right)}$, with $a,b,c,d\in k$. Then $A$ is a Frobenius algebra,
and $A$ is symmetric if and only if $u=v$ (see \cite[Example 16.66
and Exercise 16.29]{lam}). This example is due to Nakayama and
Nesbitt. Let

 $$X={\tiny \left(
\begin{array}{cccc}
0&1&0&0\\
0&0&0&0\\
0&0&0&v\\
0&0&0&0
\end{array}
\right)}, Y={\tiny \left(
\begin{array}{cccc}
0&0&1&0\\
0&0&0&u\\
0&0&0&0\\
0&0&0&0
\end{array}
\right)}, Z={\tiny \left(
\begin{array}{cccc}
0&0&0&1\\
0&0&0&0\\
0&0&0&0\\
0&0&0&0
\end{array}
\right)}$$

We have that $X^2=Y^2=Z^2=XZ=ZX=YZ=ZY=0$, $XY=uZ$ and $YX=vZ$,
therefore $A$ is a ${\bf Z}_4$-graded algebra with grading given
by $A_{\hat{0}}=kI_4$, $A_{\hat{1}}=kX$, $A_{\hat{2}}=kY$,
$A_{\hat{3}}=kZ$. Since $A_{\hat{3}}X=0$, $A_{\hat{3}}Y=0$ and
$A_{\hat{1}}X=0$, we see that $A$ is not $g$-faithful if $g\in \{
\hat{0}, \hat{1}, \hat{2}\}$. On the other hand, $A$ is
$\hat{3}$-faithful. Then $A$ is $\hat{3}$-graded Frobenius, but it
is not $g$-graded Frobenius if $g\in \{ \hat{0}, \hat{1},
\hat{2}\}$. If $u=v$, then $A$ is symmetric, but not graded
symmetric.
\end{example}

The following shows that the condition that $A_e$ is local cannot
be dropped in Theorem \ref{teoremalocal}.

\begin{example} \label{exemplulocal}
Let $R$ be a finite dimensional algebra which is decomposable as
an algebra, i.e. $R=R_1\times R_2$ for some algebras $R_1$ and
$R_2$. Then $R^*=R_1^*\oplus R_2^*$, and we have that
$R_2R_1^*=R_1^*R_2=0$ and $R_1R_2^*=R_2^*R_1=0$ in the
$R,R$-bimodule structure of $R^*$. Let $A=R\oplus R^*$ be the
trivial extension, which is an algebra as in Example
\ref{extrivext}. We consider another grading on $A$. More
precisely, $A$ is  ${\bf Z}_3$-graded with $A_{\hat{0}}=R$,
$A_{\hat{1}}=R_1^*$ and $A_{\hat{2}}=R_2^*$. Then $A$ is not
$\hat{0}$-faithful since $A_2f=0$ for any $f\in
R_1^*=A_{\hat{1}}$. Also, $A$ is neither $\hat{1}$-faithful since
$A_{\hat{2}}f=0$ for any $f\in R_2^*=A_{\hat{2}}$, nor
$\hat{2}$-faithful since $A_{\hat{1}}f=0$ for any $f\in
R_1^*=A_{\hat{1}}$. Therefore $A$ is not $\sigma$-faithful, and by
Theorem \ref{caractgrFrobenius} $A$ cannot be $\sigma$-graded
Frobenius for any $\sigma$. On the other hand $A$ is a symmetric
algebra, so it is also Frobenius.

Let us note that in fact $A$ may be regarded as a $G$-graded
algebra for any group $G$ of order at least 3. Indeed, if $e$ is
the neutral element of $G$, and $u,v$ are two other different
elements of $G$, then $A_e=A_{\hat{0}}$, $A_u=A_{\hat{1}}$ and
$A_v=A_{\hat{2}}$ defines such a grading.
\end{example}

\begin{example}
(1) Let $A=M_n(k)$ be a matrix algebra. There is a special class
of gradings on $A$, called good gradings, for which any matrix
unit $e_{ij}$ is a homogeneous element. Under certain conditions
on $n,k$ and $G$, any $G$-grading on $A$ is isomorphic to a good
grading. We have that any good grading on $A$ makes $A$ a graded
symmetric algebra. Indeed, the trace map $tr:A\rightarrow k$
satisfies $tr(xy)=tr(yx)$ for any $x,y\in A$, and it is easy to
see that $tr(yx)=0$ for any $y\in A$ implies that $x=0$ (just take
$y=e_{ij}$ for all $i,j$), so the kernel of $tr$ does not contain
non-zero left ideals. Since all $e_{ii}$'s are homogeneous
elements of degree $e$ (the neutral element of $G$), we see that
$tr(xy)=0$ if $x\in A_\sigma, y\in A_\tau$ with $\sigma\tau\neq
e$.

(2) There is another special type of gradings on $A=M_n(k)$,
called fine gradings. These are gradings for which any homogeneous
component of $A$ has dimension at most 1. Such a grading makes $A$
a graded division algebra. By Theorem \ref{grdivringsim} we see
that $A$ is graded symmetric.

(3) If $G$ is abelian and $k$ is algebraically closed, using  (1)
and (2) above and the result of \cite{bsz}, which describes any
$G$-grading on  $A=M_n(k)$ as a tensor product of a matrix algebra
with a good grading and a matrix algebra with a fine grading, we
obtain that any $G$-grading on $A$ makes it a graded symmetric
algebra.
\end{example}

\section{Frobenius functors and applications to (graded) Frobenius
algebras} \label{sfunctors}

A functor $F:{\cal A}\rightarrow {\cal B}$ is called Frobenius if
it is a left and a right adjoint of the same functor $G:{\cal
B}\rightarrow {\cal A}$. Frobenius functors originate in the work
of Morita. A functor defining an equivalence of categories is
Frobenius, and the composition of two Frobenius functors is a
Frobenius functor. It is known that a finite dimensional algebra
$A$ is Frobenius if and only if the restriction of scalars functor
$A-mod\rightarrow k-mod$ is a Frobenius functor, see for example
\cite[Theorem 28]{cmz}.

It is well known that if $A$ is a Frobenius algebra, then the
matrix algebra $M_n(A)$ is Frobenius, too. This follows from the
fact that $M_n(A)\simeq A\otimes M_n(k)$, a tensor product of
Frobenius algebras. We note that this also follows as an
application of Frobenius functors. Indeed, consider the following
commutative (up to a natural isomorphism) diagram of categories
and functors.

$$\xymatrix{
A-mod \ar[r]^>>>>{F}\ar[d]_{U_1} &  M_n(A)-mod \ar[d]^{U_2} \\
k-mod \ar[r]^{(-)^n} & k-mod }$$ where $U_1$ and $U_2$ are the
restriction of scalars functors, $F=Col_n(-)$ is the standard
equivalence of categories, and $(-)^n$ is the functor taking a
vector space $V$ to the direct sum of $n$ copies of $V$, and
acting accordingly on morphisms. It is clear that $(-)^n$ is a
Frobenius functor, since it is natural isomorphic to a functor of
the form $V\otimes (-)$, where $V$ is a vector space of dimension
$n$, and this has the left and right adjoint $V^*\otimes (-)$. Now
if $A$ is Frobenius, then $U_1$ is a Frobenius functor, so then is
$(-)^n U_1$. Then $U_2F$ is a Frobenius functor, and since $F$ is
an equivalence, we obtain that $U_2$ is also Frobenius. Thus
$M_n(A)$ is a Frobenius algebra.

We show that the converse is also true, i.e. if $M_n(A)$ is a
Frobenius algebra, then so is $A$. We first need the following.

\begin{proposition} \label{cardual}
Let $A$ be an algebra and let $M$ be a left $A$-module. If $n$ is
a positive integer, then there is an isomorphism
$Hom_{M_n(A)}(Col_n(M),M_n(A))\simeq Row_n(Hom_A(M,A))$ of right
$M_n(A)$-modules.
\end{proposition}
\begin{proof}
Let $f:Col_n(M)\rightarrow M_n(A)$ be a morphism of left
$M_n(A)$-modules. Let $m\in M$ and $X(m)$ the column having $m$ on
the first spot, and 0 elsewhere. Since $e_{11}X(m)=X(m)$, we have
that $f(X(m))=f(e_{11}X(m))=e_{11}f(X(m))$, showing that $f(X(m))$
has zeroes on any row except the first one. Thus there are
$f_1(m),\ldots,f_n(m)\in A$ such that
$$f(X(m))=\left(
\begin{array}{cccc}
f_1(m)&f_2(m)&\ldots&f_n(m)\\
0&0&\ldots&0\\
\ldots&\ldots&\ldots&\ldots\\
0&0&\ldots&0
\end{array}
\right)$$ Since $f$ is a morphism of left $M_n(A)$-modules, it is
easy to see that $f_1,\ldots,f_n\in Hom_A(M,A)$. Now $\tiny{
\left(
\begin{array}{c}
m_1\\
m_2\\
\ldots\\
m_n
\end{array}
\right)}=e_{11}X(m_1)+e_{12}X(m_2)+\ldots +e_{1n}X(m_n)$, so then
 \bea f\left(\left(
\begin{array}{c}
m_1\\
m_2\\
\ldots\\
m_n
\end{array}
\right) \right)&=&e_{11}f(X(m_1))+e_{12}f(X(m_2))+\ldots
+e_{1n}f(X(m_n))\\&=&\left(
\begin{array}{cccc}
f_1(m_1)&f_2(m_1)&\ldots&f_n(m_1)\\
f_1(m_2)&f_2(m_2)&\ldots&f_n(m_2)\\
\ldots&\ldots&\ldots&\ldots\\
f_1(m_n)&f_2(m_n)&\ldots&f_n(m_n)
\end{array}
\right) \eea Now define
$$\phi:Hom_{M_n(A)}(Col_n(M),M_n(A))\rightarrow
Row_n(Hom_A(M,A)), \phi (f)=(f_1,\ldots, f_n)$$ It is
straightforward to check that $\phi$ is a morphism of right
$M_n(A)$-modules. Obviously the kernel of $\phi$ is zero, so
$\phi$ is injective. To see that $\phi$ is surjective, let
$u_1,\ldots,u_n\in Hom_A(M,A)$, and define $g:Col_n(M)\rightarrow
M_n(A)$ by $$g\left(\left(
\begin{array}{c}
m_1\\
m_2\\
\ldots\\
m_n
\end{array}
\right) \right)\left(
\begin{array}{cccc}
u_1(m_1)&u_2(m_1)&\ldots&u_n(m_1)\\
u_1(m_2)&u_2(m_2)&\ldots&u_n(m_2)\\
\ldots&\ldots&\ldots&\ldots\\
fu_1(m_n)&u_2(m_n)&\ldots&u_n(m_n)
\end{array}
\right)$$ Then it is easy to check that $g$ is a morphism of left
$M_n(A)$-modules and $\phi(g)=(u_1,\ldots,u_n)$, so $\phi$ is also
surjective.
\end{proof}

Now we can prove the following.

\begin{proposition} \label{matriceFrobenius}
Let $A$ be a finite dimensional algebra such that $M_n(A)$ is a
Frobenius algebra for some positive integer $n$. Then $A$ is a
Frobenius algebra.
\end{proposition}
\begin{proof}
Since $M_n(A)$ is Frobenius, it is left self-injective, i.e. it is
quasi-Frobenius. This property is invariant under Morita
equivalence, so we have that $A$ is also quasi-Frobenius.

On the other hand, let $M$ be a finite dimensional left
$A$-module. Then $Col_n(M)$ is a finite dimensional left
$M_n(A)$-module. Since $M_n(A)$ is Frobenius, \cite[Theorem
16.34]{lam} shows that $$dim_k\; Hom_{M_n(A)}(Col_n(M),M_n(A))=
dim_k\; Col_n(M)$$ By Proposition \ref{cardual} we have that
$dim_k\; Hom_{M_n(A)}(Col_n(M),M_n(A))=dim_k\; Row_n(Hom_A(M,A))=
ndim_k Hom_A(M,A)$. Since $dim_k\; Col_n(M)=ndim_k\; M$, we obtain
that $dim_k Hom_A(M,A)=dim_k\; M$. Now $A$ is Frobenius by
\cite[Theorem 16.33]{lam}.
\end{proof}

We will present a graded version of the result characterizing the
Frobenius property by Frobenius functors. Let $A=\bigoplus
\limits_{g\in G}A_g$ be a $G$-graded finite dimensional algebra.
We can regard the basic field $k$ as a $G$-graded algebra with
trivial grading, i.e. the homogeneous component of degree $e$ is
the whole of $k$ and all the other homogeneous components are
zero. Let $k-gr$ be the associated category of graded modules,
which is in fact the category of $G$-graded vector spaces. A
$G$-graded vector space is just a vector space which is a direct
sum of subspaces indexed by $G$. Let $U:A-gr\rightarrow k-gr$ be
the forgetful functor, which forgets the $A$-action, but preserves
the $G$-grading.

We define the functor $F:k-gr\rightarrow A-gr$ by $F(V)=A\otimes
V$, with $A$-action on the first tensor position, and grading
given by $F(V)_g=\bigoplus \limits_{\sigma\tau=g}A_{\sigma}\otimes
V_{\tau}$ for any $g\in G$. On morphisms $F$ acts as the tensor
product with $Id_A$.

We also define the functor $T:k-gr\rightarrow A-gr$ by
$T(V)=Hom(A,V)$, with the $A$-action induced by the right
$A$-module structure of $A$, and the grading given by
$T(V)_\sigma=\{ f\in Hom(A,V)\; |\; f(A_g)\subseteq V_{g\sigma}
\mbox{ for any }g\in G\}$. On morphisms $T$ acts as $Hom(Id_A,-)$.
It is straightforward to check  that $F$ is a left adjoint of $U$,
and $T$ is a right adjoint of $U$. Now we have the following
characterization of the graded Frobenius property.

\begin{proposition}
$A$ is graded Frobenius if and only if the forgetful functor
$U:A-gr\rightarrow k-gr$ is a Frobenius functor.
\end{proposition}
\begin{proof}
$U$ is a Frobenius functor if and only if its left and right
adjoints, $F$ and $T$, are naturally isomorphic. Assume that $U$
is Frobenius, thus there is an isomorphism $\eta(V):A\otimes
V\rightarrow Hom(A,V)$ of graded $A$-modules, natural in $V$, for
each graded vector space $V$. In particular, for $V=k$, with
grading concentrated in the degree $e$-component, we obtain an
isomorphism $A\otimes k\simeq Hom(A,k)$ of graded $A$-modules.
Since $A\otimes k\simeq A$, this yields an isomorphism of graded
left $A$-modules $A\simeq A^*$, i.e. $A$ is graded Frobenius.

Conversely, assume that $A$ is graded Frobenius, and let
$\theta:A\rightarrow A^*$ be an isomorphism in $A-gr$. Then for
any $V\in k-gr$ define $\eta(V)$ as the composition
$$\eta(V): A\otimes V\stackrel{\theta\otimes I}{\longrightarrow} A^*\otimes
V\stackrel{\gamma(V)}{\longrightarrow} Hom(A,V)$$ where
$\gamma(V)$ is the natural linear isomorphism defined by
$\gamma(V)(f\otimes v)(r)=f(r)v$ for any $f\in A^*$, $v\in V$ and
$r\in A$. It is straightforward to check that $\gamma(V)$ is in
fact an isomorphism of graded $A$-modules, and since $\theta
\otimes I$ is  an isomorphism in $A-gr$, we obtain that $\eta(V)$
is also an isomorphism in $A-gr$. Obviously $\eta(V)$ is natural
in $V$, and this shows that $\eta$ is a natural isomorphism
between $F$ and $T$.
\end{proof}

If $A$ is a graded algebra, another functor providing good
information about $A$ is  $(-)_e:A-gr\rightarrow A_e-mod$. For
instance, $A$ is strongly graded if and only if $(-)_e$ is an
equivalence of categories. We investigate whether $(-)_e$ provides
information about Frobenius properties on $A$. In this direction
we have the following.

\begin{theorem}
Let $A$ be a finite dimensional $G$-graded algebra. The following
assertions are true.\\
(1) If $A$ is graded Frobenius and $A$ is a projective left
$A_e$-module, then $(-)_e$ is a Frobenius functor.\\
(2) If $(-)_e$ is a Frobenius functor and $A_e$ is a Frobenius
algebra, then $A$ is graded Frobenius.
\end{theorem}
\begin{proof}
(1) Assume that $A$ is graded Frobenius. By the proof of Theorem
\ref{thcaractgrFrob} (for $\sigma =e$), we see that
$\mu:A\rightarrow Hom_{A_e}(A,A_e)$, $\mu
(r_{\sigma})(a)=a_{\sigma^{-1}}r_{\sigma}$ for any $\sigma \in G,
r_{\sigma}\in A_{\sigma}$ and $a\in A$, is an isomorphism in
$A-gr$. Moreover \bea \mu
(r_{\sigma}b_e)(a)&=&a_{\sigma^{-1}}r_{\sigma}b_e\\
&=&\mu (r_{\sigma})(a)b_e\\
&=&(\mu (r_{\sigma}b_e)(a)\eea for any $b_e\in A_e$, showing that
$\mu$ is also a morphism of right $A_e$-modules. Now we use
\cite[Theorem 3.4]{mn} and find that $(-)_e$ is a Frobenius
functor.

(2) We have seen that a right adjoint of $(-)_e$ is $Coind$. A
left adjoint of $(-)_e$ is the induced functor
$Ind:A_e-mod\rightarrow A-gr$, $Ind(N)=A\otimes _{R_e}N$, with
grading $Ind(N)_g=A_g\otimes _{A_e}N$, and acting on morphisms as
$Id_A\otimes_{A_e}-$. There is a natural transformation
$\eta:Ind\rightarrow Coind$ defined by $$\eta(N):Ind(N)\rightarrow
Coind(N), \eta (N)(r\otimes x)(b)=\sum_g (b_{g^{-1}}r_g)x \mbox{
for }r,b\in A,x\in N$$ By \cite[Theorem 3.4]{mn} the functor
$(-)_e$ is Frobenius if and only if $\eta(A_e)$ is an isomorphism.
Since $Ker(\eta(A_e))=t_e(A)$, we see that $A$ is $e$-faithful.
Now $A$ is graded Frobenius by Theorem \ref{thcaractgrFrob}.

\end{proof}

For many ring properties "P" it is true that for an algebra $A$
graded by a finite group $G$ we have that $A$ satisfies the graded
version of the property "P" if and only if the smash product $A\#
(kG)^*$ satisfies "P". If "P" is the property of being Frobenius,
then this connection does not hold. The reason is that the
Frobenius property is not a categorial one, more precisely it is
not invariant under Morita equivalence, despite that it is
preserved between $A$ and $M_n(A)$. In fact it was shown in
\cite{bergen} that for a finite dimensional Hopf algebra $H$
acting on a finite dimensional algebra $A$, we have that the smash
product $A\# H$ is Frobenius if and only if $A$ is Frobenius. Now
we give a new short proof of this result, by using Frobenius
functors.

\begin{theorem}
Let $H$ be a finite dimensional Hopf algebra and let $A$ be a
finite dimensional left $H$-module algebra. Then $A\# H$ is
Frobenius if and only if $A$ is Frobenius.
\end{theorem}
\begin{proof}
 Since $H$ is finite dimensional, $A$ is a right $H^*$-comodule algebra, and we can consider
 the category of generalized Hopf modules $ _A{\cal M}^{H^*}$. There is an isomorphism of
 categories $P:A\# H-mod\rightarrow \, _A{\cal M}^{H^*}$, which
 associates to a left $A\# H$-module $M$ the space $M$ with the
 left $A$-module structure obtained by restriction of scalars via
 $A\subseteq A\# H$, and the right $H^*$-comodule structure coming
 from the left $H$-module structure. If $U:\,  _A{\cal
 M}^{H^*}\rightarrow A-mod$ is the functor forgetting the
 $H^*$-comodule structure, then $UP$ is just the restriction of
 scalars, and it is known that $U$ is a Frobenius functor.

 If $A$ is Frobenius, then the restriction of scalars
 $U_1:A-mod\rightarrow k-mod$ is a Frobenius functor, so then
 $U_1UP$ is a Frobenius functor, as a composition of Frobenius
 functors. But $U_1UP:A\# H-mod\rightarrow k-mod$ is just the restriction of
 scalars, and we obtain that $A\# H$ is Frobenius.

 Conversely, assume that $A\# H$ is Frobenius. Then $A\# H$ is a
 left $H^*$-module algebra, and by the above considerations $(A\#
 H)\# H^*$ is Frobenius. Now $(A\#
 H)\# H^*\simeq M_n(R)$, where $n=dim(H)$, by the duality theorem
 for finite dimensional Hopf algebras (see \cite[Corollary 9.4.17]{mo}), and we
 obtain that $A$ is Frobenius by Proposition
 \ref{matriceFrobenius}.

\end{proof}

We note that if $A$ is a $G$-graded algebra and $X$ is a finite
$G$-set, then there is an associated smash product $A\# X$, see
\cite[page 216]{nvo}. We have that if $A$ is Frobenius, then so is
$A\# X$. Indeed, the category of left modules over $A\# X$ is
isomorphic to the category of generalized Doi-Hopf modules
$_A{\cal M}^{kX}$, which is nothing but the category of
$A$-modules graded by the $G$-set $X$. The forgetful functor
$_A{\cal M}^{kX}\rightarrow _A{\cal M}$ is Frobenius (see for
instance \cite[Theorem 54]{cmz}, and arguments similar to the ones
above can be now applied.\\

{\bf Acknowledgment} We thank Daniel Bulacu for useful discussions
about Frobenius algebras in monoidal categories. The research was
supported by the UEFISCDI Grant PN-II-ID-PCE-2011-3-0635, contract
no. 253/5.10.2011 of CNCSIS.

\end{document}